\documentclass{amsart}       

\usepackage{amsrefs}
\usepackage{graphicx}
\usepackage{amsmath}
\usepackage{amscd}
\usepackage{amsfonts}
\usepackage{amssymb}
\usepackage{color}

\theoremstyle{plain}
\newtheorem{theorem}{Theorem}[section]
\newtheorem{lemma}[theorem]{Lemma}
\newtheorem{proposition}[theorem]{Proposition}

\newtheorem{corollary}[theorem]{Corollary}

\theoremstyle{definition}
\newtheorem{definition}[theorem]{Definition}
\newtheorem{example}[theorem]{Example}
\newtheorem{remark}[theorem]{Remark}

\newcommand{\be}{\begin{equation}}
\newcommand{\ee}{\end{equation}}
\newcommand{\bes}{\begin{equation*}}
\newcommand{\ees}{\end{equation*}}

\newcommand{\cE}{\mathcal{E}}
\newcommand{\cH}{\mathcal{H}}
\newcommand{\cK}{\mathcal{K}}

\newcommand{\cB}{\mathcal{B}}
\newcommand{\cM}{\mathcal{M}}
\newcommand{\cN}{\mathcal{N}}

\newcommand{\cA}{\mathcal{A}}
\newcommand{\cR}{\mathcal{R}}
\newcommand{\cO}{\mathcal{O}}

\newcommand{\cT}{\mathcal{T}}

\newcommand{\cC}{\mathcal{C}}

\newcommand{\mb}[1]{\mathbb{#1}}

\newcommand{\alg}{\operatorname{Alg}}
\newcommand{\Aut}{\operatorname{Aut}}

\newcommand{\Mult}{\operatorname{Mult}}

\newcommand{\spn}{\operatorname{span}}
\newcommand{\id}{\operatorname{id}}

\newcommand{\rep}{\operatorname{rep}}

\newcommand{\fA}{\mathfrak A}
\newcommand{\fs}{\mathfrak s}
\newcommand{\ft}{\mathfrak t}

\begin{document}

\title[Dynamics and automorphisms of free products]{Automorphisms of free products and their application to multivariable dynamics}
\author{Christopher Ramsey}
\address{University of Virginia, Charlottesville, VA, USA}
\email{cir6d@virginia.edu}

\thanks{The author is partially supported by NSERC, Canada.}

\begin{abstract}
We examine the completely isometric automorphisms of a free product of noncommutative disc algebras. It will be established that such an automorphism is given simply by a completely isometric automorphism of each component of the free product and a permutation of the components. This mirrors a similar fact in topology concerning biholomorphic automorphisms of product spaces with nice boundaries due to Rudin, Ligocka and Tsyganov. 

This paper is also a study of multivariable dynamical systems by their semicrossed product algebras. A new form of dynamical system conjugacy is introduced and is shown to completely characterize the semicrossed product algebra. This is proven by using the rigidity of free product automorphisms established in the first part of the paper.

Lastly, a representation theory is developed to determine when the semicrossed product algebra and the tensor algebra of a dynamical system are completely isometrically isomorphic. 
\end{abstract}


\maketitle

\section{Introduction}\label{sec:intro}

In this paper we wish to study the completely isometric automorphisms of a free product of noncommutative disc algebras amalgamated over the identity. In particular it will be established that such automorphisms have a rigid structure. This will tell us that one can think of such a free product as the appropriate non-commutative analogue of an algebra of holomorphic functions.

The noncommutative disc algebra, $\fA_n$, introduced by Popescu \cite{Pop0}, is the universal operator algebra generated by a row contraction,  $(\fs_1,\cdots, \fs_n)$ such that $\sum_{i=1}^n \fs_i\fs_i^* \leq I$, and the identity. $\fA_1 = A(\mb D)$, the disc algebra, and for $n\geq 2$ this should be thought of as a multivariable noncommutative generalization of the disc algebra. This algebra is also completely isometrically isomorphic to the norm closed algebra generated by the left creation operators and the identity on the full Fock space. 

Popescu generalized this from one row contraction to several which he called a sequence of contractive operators, namely $(\fs_{i,1}, \cdots, \fs_{i,n_i}), 1\leq i \leq m$ with each tuple a row contraction \cite{Pop1, Pop2}. In this paper, such an object will be called a {\bf set of row contractions} since the author believes this to be clearer. Popescu goes on to prove that the universal operator algebra generated by a set of row contractions and the identity is the free product of $\fA_{n_1}, \cdots, \fA_{n_m}$ amalgamated over the identity. This algebra is denoted $*_{i=1}^m \fA_{n_i}$ or $\fA_{n}^{*m}$ if they all have the same size.
He also shows that these free products are completely isometrically isomorphic to some universal operator algebras studied by Blecher and Paulsen \cite{BlechPaul, Blech} which they denote $OA(\Lambda, \cR)$.

It will be established that every completely isometric automorphism of such a free product of $\fA_n$ is induced by a biholomorphism of the character space. Thus, such an automorphism will be given by a completely isometric automorphism of each noncommutative disc algebra and a permutation of these algebras. This is due to Rudin, Ligocka and Tsyganov \cite{Rudin, Lig, Tsyganov}, that a biholomorphism of a product space is just a product of biholomorphisms in each space and a permutation of the spaces. This will all be shown in Section 2.


The remainder of the paper concerns multivariable dynamics and operator algebras.
Specifically, a multivariable dynamical system $(X,\sigma)$ is a locally compact Hausdorff space $X$ along with $n$ proper continuous maps $\sigma = (\sigma_1,\cdots, \sigma_n)$ of $X$ into itself. To such a system one can construct two non-selfadjoint operator algebras. The first being the {\bf semicrossed product algebra} $C_0(X) \times_\sigma \mb F_n^+$ which is the universal operator algebra generated by $C_0(X)$ and $\mathfrak s_iC_0(X), 1\leq i\leq n$, where the operators $\mathfrak s_1,\cdots, \mathfrak s_n$ are contractions and satisfy the following covariance relations
\[
f \mathfrak s_i = \mathfrak s_i (f\circ\sigma_i) \ \ {\rm for } \ \ f\in C_0(X) \ \ {\rm and} \ \ 1\leq i\leq n.
\]
Similarly define the {\bf tensor algebra} to be the same as above but with the added condition that the generators are row contractive $\|[\mathfrak s_1 \cdots \mathfrak s_n]\| \leq 1$. These algebras should be thought of as encoding an action of the free semigroup algebra $\mb F_n^+$ on $C_0(X)$. If $n=1$ these algebras are the same and it has been proven that two such algebras are algebraically isomorphic if and only if their associated dynamical systems are conjugate, there exists a homeomorphism between the spaces intertwining the maps. This study was first initiated by Arveson \cite{Arv1} in 1967, later taken up in \cite{ArvJos, Peters, HadHoo} with the introduction of the semicrossed product given by Peters and finally proven in full generality by Davidson and Katsoulis \cite{DavKat1} in 2008.

For the multivariable case, Davidson and Katsoulis  in \cite{DavKat2} introduced a new notion of conjugacy. We say that two dynamical systems $(X,\sigma)$ and $(Y,\tau)$ are {\bf piecewise conjugate} if there is a homeomorphism $\gamma: X \rightarrow Y$ and an open cover $\{V_\alpha : \alpha\in S_n\}$ such that $\gamma \circ\tau_i \circ\gamma^{-1}|_{V_\alpha} = \sigma_{\alpha(i)}|_{V_\alpha}$. They went on to show that if the tensor or semicrossed product algebras of two systems are isomorphic then the dynamical systems are piecewise conjugate. 
A converse to this was then established by them for the tensor algebra case in many contexts, $X$ is totally disconnected or $n\leq 3$ for instance (the $n=4$ case was established later in \cite{Ramsey}), most of which gave that two algebras being algebraically isomorphic was enough to produce a completely isometric isomorphism between them. For the semicrossed product case it was shown that the converse fails as in the following example.

\begin{example}{\cite[Example 3.24]{DavKat2}}
Let $X = \{1,2\}$ and consider the following self maps on $X$:
\[
\sigma_1 = \id, \ \ \sigma_2(1) = 2, \sigma_2(2) = 1,  \ \ \tau_1(x) = 1, \ \ \tau_2(x) = 2.
\]
Hence, $(X,\sigma)$ and $(X,\tau)$ are piecewise conjugate. With some work they show that $C^*_e(C(X) \times_\sigma \mathbb F^2_+) = M_2(C^*(\mathbb F^3))$, while $C^*_e(C(X) \times_\tau \mathbb F^2_+) = \mathcal O_2$, the Cuntz algebra. These C$^*$-algebras are not isomorphic so the semicrossed products cannot be completely isometrically isomorphic.
\end{example}

However, this just implies that piecewise conjugacy is not strong enough to be an invariant for the semicrossed product algebras. If one broadens the theory to a C$^*$-dynamical system, which is comprised of a C$^*$-algebra and $n$ $*$-endomorphisms, then Kakariadis and Katsoulis have shown that the C$^*$-dynamical system forms a complete invariant for the semicrossed product algebra up to what they call outer conjugacy when the C$^*$-algebra has trivial center \cite{KakKats}. Additionally, much work has gone into describing the C$^*$-envelope of the tensor algebra and semicrossed product algebra of a dynamical system \cite{Duncan0, DavFulKak2}. The interested reader should look at the recent survey article by Davidson, Fuller and Kakariadis \cite{DavFulKak} which presents these universal operator algebras in the more general context of semicrossed products of operator algebras by semigroup actions. The study of these semigroup actions can be found in \cite{DavFulKak2} and \cite{DuncanPeters}.


To deal with the obstruction posed by the above example, Section 3 introduces a form of conjugacy, called partition conjugacy, that lies between piecewise conjugacy and conjugacy. Theorem \ref{Thm:conjugacy} shows that if two systems are partition conjugate then their semicrossed product algebras are completely isometrically isomorphic. Conversely, Section 4 shows that a c.i. isomorphism of semicrossed product algebras induces a c.i. isomorphism of free products of noncommutative disc algebras. The rigidity proven in Section 2 is then passed along which allows us to show that semicrossed product algebras are c.i. isomorphic if and only if their related dynamical systems are partition conjugate.

Finally, Section 5 introduces a representation theory that is a combination of the nest representation theory from \cite{DavKat2} and Duncan's edge-colored directed graph representation theory from \cite{Duncan2}. This allows us to distinguish the pre-images of the maps at a given point, which in turn proves that the tensor and semicrossed product algebras are c.i.i if and only if the ranges of the maps of the dynamical system are pairwise disjoint.


\section{Automorphisms of free products}

Consider the free product of noncommutative disc algebras $*_{i=1}^m \fA_{n_i}$ with generators $(\fs_{i,1}, \cdots, \fs_{i,n_i}), 1\leq i \leq m$, a set of row contractions.
Popescu \cite[Theorem 2.5]{Pop1}, by way of a dilation result, shows that the generators of the free product can be taken to be row isometries of Cuntz type, in other words each $\fs_{i,j}$ is an isometry and 
\[
\fs_{i,1}\fs_{i,1}^* + \cdots + \fs_{i,n_i}\fs_{i,n_i}^* = I, \ \ 1\leq i\leq m,
\]
in some concrete representation, $*_{i=1}^m \fA_{n_i} \subseteq B(\cH)$.
This is quite natural as the C$^*$-envelope of $\fA_d$ is the Cuntz algebra $\cO_d$, with $\cO_1 = C(\mb T)$.

Let $\cC$ be the commutator ideal of $\fA_n$, the closed ideal generated by $\fs_i\fs_j - \fs_j\fs_i, 1\leq i,j\leq n$. Then $\cA_n = \fA_n/\cC$ is the universal operator algebra generated by a commuting row contraction and the identity. 

A reproducing kernel Hilbert space (RKHS) $\cH$ is a Hilbert space of functions on a specific topological space $X$ such that at every $\lambda\in X$ point evaluation at $\lambda$ is a continuous linear functional. This implies that there is a kernel function $k_\lambda \in \cH$ such that $\langle f, k_\lambda\rangle = f(\lambda)$, so it reproduces the value of $f$ at $\lambda$. The RKHS is uniquely determined by its kernel functions and has a natural identification  with certain holomorphic functions on $X$. Finally, the multiplier algebra, $\Mult(\cH)$, is the operator algebra of all functions on $X$ given by
\[
\Mult(\cH) = \{M_f : M_f(h) = fh \in \cH, \forall h\in \cH\}.
\]
The Drury-Arveson space, $\cH^2_n$, is the RKHS on the open unit ball in $\mb C^n$, namely $\mb B_n$, with kernel functions 
\[
k_\lambda(z) \ = \ \frac{1}{1 - \langle z,\lambda\rangle}, \ \ \lambda \in \mb B_n.
\]
Because $1\in \cH^2_n$ then for all $M_f\in \cM_n =: \Mult(\cH^2_n)$ we have that $M_f(1) = f\cdot 1 = f\in \cH^2_n$. Hence, the multipliers are naturally identified as holomorphic functions. One should note though, for $n\geq 2, \cM_n$ is contained in but not equal to $H^\infty(\mb B_n)$ since the operator norm on $\cM_n$ is not equal to the supremum norm. Arveson \cite{Arv2} developed this theory and showed by way of a dilation theory for commuting row contractions that $\cA_n = \fA_n/\cC$ is completely isometrically isomorphic to the norm closed algebra $\overline{\alg}\{1, M_{z_1}, \cdots, M_{z_n}\}$. Lastly, this implies that $\cA_n$ is naturally contained in $C(\overline{\mb B}_n) \cap H^\infty(\mb B_n)$ but not equal to when $n\geq 2$.

The development of the Drury-Arveson space was spurred on by the obstructions one faces in the polydisc $\mb D^n$. Specifically, while Ando \cite{Ando} showed that two commuting contractions dilate to two commuting unitaries, giving that the  bidisc algebra $A(\mb D^2) = C(\overline{\mb D}^2)\cap H(\mb D^2)$ is the universal operator algebra generated by two commuting contractions, Parrott and later Kajser and Varopoulos \cite{Parrott, Varopoulos} showed that there are three commuting contractions that do not dilate to three commuting isometries. Hence, the universal operator algebra generated by $n$ commuting contractions when $n\geq 3$ is a mysterious algebra that is not the polydisc algebra $A(\mb D^n)$. However, Ito \cite[Proposition I.6.2]{NagyFoias} proved that a set of commuting isometries dilates to a commuting set of unitaries and so the polydisc algebra is the universal operator algebra generated by $n$ commuting isometries.

For the context of a set of row contractions we would like to develop an operator algebra of holomorphic functions on $\times_{i=1}^m \mb B_{n_i}$, which one could call the {\bf polyball}.

\begin{definition}
For $n_1,\cdots, n_m \in \mb N$ let $\cH^2_{n_1,\cdots, n_m}$ be the RKHS on $\times_{i=1}^m \mb B_{n_i}$ given by the kernel functions
\[
k_\lambda(z) = \prod_{i=1}^m \frac{1}{1 - \langle z_i, \lambda_i\rangle_{\mb C^{n_i}}}, \ \ \lambda \in \times_{i=1}^m \mb B_{n_i}.
\]
Thus, $\cH^2_{n_1,\cdots,n_m} = \cH^2_{n_1}\otimes \cdots\otimes \cH^2_{n_m}$ and is then naturally identified as a space of holomorphic functions on the polyball.
As well, if $M_{z_{i,j}}, 1\leq j\leq n_i, 1\leq i\leq m$ are the coordinate multipliers in $\Mult(\cH^2_{n_1,\cdots, n_m})$ then define
\[
\cA_{n_1,\cdots, n_m} \ = \ \overline{\alg}\{1, M_{z_{1,1}}, \cdots, M_{z_{m, n_m}}\}.
\]
\end{definition}

As before, since $1\in \cH^2_{n_1,\cdots, n_m}$ then $\cA_{n_1,\cdots, n_m}$ can also be naturally identified as an operator algebra of holomorphic functions. Furthermore, the generators of $\cA_{n_1,\cdots, n_m}$ form a set of row contractions and so by the universal property there exists a completely contractive homomorphism 
\[
\pi : *_{i=1}^m \fA_{n_i} \rightarrow \cA_{n_1,\cdots, n_m}
\]
mapping $\fs_{i,j}$ to $M_{z_{i,j}}$.
Except in special cases, $\ker \pi$ will not be equal to the commutator ideal $\cC$ of the free product because $*_{i=1}^m \fA_{n_i}/\cC$ is the universal operator algebra generated by a commuting set of commuting row contractions.

Going back to the noncommutative context, recall that the character space of a Banach algebra $\cB$, denoted $M(\cB)$, is the space of all multiplicative linear functionals endowed with the weak$^*$ topology. Popescu \cite{Pop1} proved that the characters of $*_{i=1}^m \fA_{n_i}$ are exactly the point evaluations $\rho_\lambda$ where $\lambda \in \times_{i=1}^m \overline{\mb B}_{n_i}$ and $\rho_\lambda(\fs_{i,j}) = \lambda_{i,j}$. If $\varphi \in \Aut(*_{i=1}^m \fA_{n_i})$, a bounded automorphism, then it induces a homeomorphism of the character space by 
\[
\rho_\lambda \in M(*_{i=1}^m \fA_{n_i}) \ \ \mapsto \ \ \varphi^*(\rho_\lambda) := \rho_\lambda \circ \varphi.
\]
Since the boundary of the closed polyball will be mapped to itself, $\varphi^*$ is also a homeomorphism of the open polyball.
Thus, for every $\lambda\in \times_{i=1}^m \mb B_{n_i}$ there exists $\mu\in \times_{i=1}^m \mb B_{n_i}$ such that $\varphi^*(\rho_\lambda) = \rho_\mu$ and so we abuse notation and define $\varphi^*$ to be a map from $\times_{i=1}^m \mb B_{n_i}$ into itself defined by 
\[
\varphi^*(\lambda) = \mu \ \ \Leftrightarrow \ \ \varphi^*(\rho_\lambda) = \rho_\mu,
\]
which is a homeomorphism on the polyball. Notice that $M(\cA_{n_1,\cdots, n_m})$ is also homeomorphic to $\times_{i=1}^m \overline{\mb B}_{n_i}$.

\begin{proposition}\label{Prop:characterextension}
Let $\pi : *_{i=1}^m \fA_{n_i} \rightarrow \cA_{n_1,\cdots, n_m}$ be the canonical quotient map, then 
\[
\ker \pi = \bigcap_{\lambda\in \times_{i=1}^m \mb B_{n_i}} \ker \rho_\lambda.
\]
\end{proposition}
\begin{proof}
Let $\lambda\in \times_{i=1}^m \mb B_{n_i}$ and $\rho_\lambda^{\cA} \in M(\cA_{n_1,\cdots,n_m})$. Then $\rho_\lambda^{\cA} \circ\pi$ is a multiplicative linear functional, i.e. a character, of $*_{i=1}^m \fA_{n_i}$. Hence, there exists $\mu \in \times_{i=1}^m \mb B_{n_i}$ such that $\rho_\lambda^\cA\circ \pi = \rho_\mu$.
However, 
\[
\lambda_{i,j} = \rho_\lambda^\cA(M_{z_{i,j}}) = \rho_\lambda^\cA\circ\pi(\fs_{i,j}) = \rho_\mu(\fs_{i,j}) = \mu_{i,j}.
\]
Thus, $\lambda = \mu$ and so $\ker \pi \subset \ker \rho_\lambda$ for all $\rho_\lambda \in M(*_{i=1}^m \fA_{n_i})$.

Conversely, if $A\in *_{i=1}^m \fA_{n_i}$ such that $\rho_\mu(A) = 0$ for every $\mu$ in the polyball then $\pi(A) \in \ker \rho_\lambda^\cA$ for every $\lambda$ in the polyball. But $\pi(A) = M_f$ where $f$ is a continuous function on $\times_{i=1}^m \overline{\mb B}_{n_i}$ with 
\[
0 =\rho_\lambda^\cA(\pi(A)) = \rho_\lambda^\cA(M_f) = f(\lambda).
\]
Therefore, $f = 0$ and $\pi(A) = 0$.
\end{proof}

\begin{corollary}\label{Cor:holomorphicautos}
If $\varphi \in \Aut(*_{i=1}^m \fA_{n_i})$ then it induces $\tilde\varphi \in \Aut(\cA_{n_1,\cdots,n_m})$, such that $\tilde\varphi \circ\pi = \pi \circ\varphi$.
\end{corollary}
\begin{proof} 
By the previous proposition, since $\varphi^*$ is a bijection of $\times_{i=1}^m \mb B_{n_i}$, we have 
\[
\varphi^{-1}(\ker \pi) \ = \ \varphi^{-1}\big(\bigcap_{\lambda\in \times_{i=1}^m \mb B_{n_i}} \ker \rho_\lambda\big) 
\ = \ \bigcap_{\lambda \in \times_{i=1}^m \mb B_{n_i}} \ker \varphi^*(\rho_\lambda)
\]
\[
= \ \bigcap_{\lambda \in \times_{i=1}^m \mb B_{n_i}} \ker \rho_\lambda \ = \ \ker \pi.
\]
Therefore, $\varphi$ induces the required automorphism of the quotient by $\pi$.
\end{proof}

From this theory we can now prove that $\varphi^*$ is not just a homeomorphism but also a biholomorphism by a standard argument using the character space. 

\begin{theorem}
If $\varphi \in \Aut(*_{i=1}^m \fA_{n_i})$ then $\varphi^*$ is a biholomorphism and so is in $\Aut(\times_{i=1}^m \mb B_{n_i})$.
\end{theorem}
\begin{proof}
By Corollary \ref{Cor:holomorphicautos} there exists $\tilde\varphi \in \Aut(\cA_{n_1,\cdots, n_m})$ induced by $\varphi$. Define the following map $H:\times_{i=1}^m \mb B_{n_i} \rightarrow \times_{i=1}^m \mb C^{n_i}$ by 
\[
H(\lambda) = (\tilde\varphi\circ\pi(\fs_{1,1})(\lambda), \cdots, \tilde\varphi\circ\pi(\fs_{m,n_n})(\lambda)), \ \ \lambda\in \times_{i=1}^m \mb B_{n_i}.
\]
This is a holomorphic map on $\times_{i=1}^m \mb B_{n_i}$ into itself because each $\tilde\varphi\circ\pi(\fs_{i,j})$ is a holomorphic function. 
Moreover, by the proof of Proposition \ref{Prop:characterextension}
\[
H(\lambda) = (\rho_\lambda^\cA(\tilde\varphi\circ\pi(\fs_{1,1})), \cdots, \rho_\lambda^\cA(\tilde\varphi\circ\pi(\fs_{m,n_m})))
\]
\[
= (\rho_\lambda^\cA\circ\pi\circ\varphi(\fs_{1,1}), \cdots, \rho_\lambda^\cA\circ\pi\circ\varphi(\fs_{m,n_m})
\]
\[
= (\varphi^*(\rho_\lambda(\fs_{1,1})), \cdots, \varphi^*(\rho_\lambda(\fs_{m,n_m}))) = \varphi^*(\lambda) \in \times_{i=1}^m \mb B_{n_i},
\]
because $\varphi^*(\rho_\lambda)$ is completely contractive, it's a character.
Do this again for $\varphi^{-1}$ to get that $\varphi^*$ is a biholomorphism.
\end{proof}

It remains to be shown that the map $\varphi \mapsto \varphi^*$ is a bijection. To this end we first look at the structure of the automorphism group of the polyball.

It is well known that $\Aut(\mb D^n) = \Aut(\mb D)^n \rtimes S_n$, where the symmetric group has the natural permutation action, see \cite{Rudin}. A generalization of this due to Ligocka and Tsyganov \cite{Lig, Tsyganov} says that if there is biholomorphism of a product space onto another product space, with each space having a $C^2$ boundary, then there are the same number of spaces and by rearranging the order of the second product space the biholomorphism is just a product of individual biholomorphisms.
Specifically, the automorphisms of $\times_{i=1}^m \mb B_{n_i}$ are given by automorphisms in each of the spaces along with permutations among the homeomorphic spaces. Specifically, 
\[
 \Aut(\times_{i=1}^n (\mb B_i)^{k_i}) = \times_{i=1}^n (\Aut(\mb B_i)^{k_i} \rtimes S_{k_i}),
\]
where the automorphisms of the ball are fractional linear transformations \cite{Rudin2}.

In \cite{DavPitts, Pop3}, Davidson-Pitts and Popescu separately proved that the isometric automorphisms of $\fA_n$ are completely isometric and in bijective correspondence with the automorphisms of the ball, $\mb B_n$, in a natural way. Indeed, both papers show that the following construction of Voiculescu \cite{Voic} is the appropriate way to create an automorphism of $\fA_n$ from an automorphism of the ball.

Let $\phi \in \Aut(\mb B_n)$. Then there exists $X = \left[\begin{array}{cc} x_0 & \eta_1^* \\ \eta_2 & X_1 \end{array}\right] \in U(1,n)$ such that
\[
\phi(\lambda)  \ \ = \ \ \frac{X_1\lambda + \eta_2}{x_0 + \langle \lambda, \eta_1\rangle} \ \ for \ \ \lambda \in \mb B_n.
\]
Voiculescu then defines an automorphism of the Cuntz-Toeplitz algebra $\cE_n$, with generators the left creation operators, $L_1,\cdots, L_n$, on the full Fock space, by 
\[
\varphi(L_\zeta) = (\overline{x}_0 I - L_{\overline{\eta}_2})^{-1}(L_{\overline{X}_1\zeta} - \langle \zeta, \overline{\eta}_1\rangle I ),
\]
where $L_\zeta = \sum_{i=1}^n \zeta_i L_i$. One can think of the above map as a non-commutative fractional linear map. This automorphism restricts to the non-commutative disc algebra, $\fA_n$. Lastly, this is the natural way to do this because $\varphi^* = \phi$, \cite[Lemma 4.10]{DavPitts}. In what follows, note that $\Aut_{ci}(\cB)$ denotes the completely isometric automorphisms of $\cB$.

\begin{lemma}\label{Lemma:surjective}
Let $\phi \in \Aut(\times_{i=1}^m \mb B_{n_i})$. There exists $\varphi \in \Aut_{ci}(*_{i=1}^m \fA_{n_i})$ such that $\varphi^* = \phi$.
\end{lemma}
\begin{proof}
As was described before any automorphism of the polyball is the composition of a permutation and automorphisms in each component. 

First, consider a permutation $\alpha \in S_{n_1\cdots n_m}$ such that $n_{\alpha(i)} = n_i$ for all $1\leq i\leq m$. Define the map $\varphi_\alpha$ on the generators of the free product by
\[
\varphi_\alpha(\fs_{i,j}) = \fs_{\alpha(i),j}, \ \ 1\leq j\leq n_i, 1\leq i\leq m.
\]
By the universal property of the free product this extends to a completely contractive endomorphism of $*_{i=1}^m \fA_{n_i}$. As well, $\varphi_{\alpha^{-1}}$ is the inverse of $\varphi_\alpha$ on the generators and so $\varphi_{\alpha^{-1}} = \varphi_\alpha^{-1}$. Thus, $\varphi_\alpha \in \Aut_{ci}(*_{i=1}^m \fA_{n_i})$ and $\varphi_{\alpha}^* = \alpha^{-1}$.

Second, for a fixed $1\leq k\leq m$ consider $\tilde\varphi \in \Aut_{ci}(\fA_{n_k})$. Then define $\varphi$ on the generators by
\[
\varphi(\fs_{i,j}) = \fs_{i,j}, 1\leq j\leq n_i, i\neq k, \ \ {\rm and} \ \ \varphi|_{\fA_{n_k}} = \tilde\varphi.
\]
Again by the universal property $\varphi$ extends to a completely contractive homomorphism of $*_{i=1}^m \fA_{n_i}$. In the same way construct a c.c. homomorphism from $\tilde\varphi^{-1}$ to see that this $\varphi\in \Aut_{ci}(*_{i=1}^m \fA_{n_i})$ with $\varphi^*|_{|\mb B_{n_i}} = \id, i\neq k$ and $\varphi^*|_{\mb B_{n_k}} = \tilde\varphi^*$. Therefore by the fact that $\Aut_{ci}(\fA_{n_k}) \simeq \Aut(\mb B_{n_k})$ the conclusion follows.
\end{proof}


Now we turn our attention to proving the injectivity of the map $\varphi \mapsto \varphi^*$. First we need a few results. Consider the map $\gamma_{\bf z}$ on $*_{i=1}^m \fA_{n_i}$, where ${\bf z} = (z_1,\cdots, z_m), z_i = (z_{i,1},\cdots, z_{i,n_i}) \in \overline{\mb D}^{n_i}$ and
\[ 
\gamma_{\bf z}(\fs_{i,j}) = z_{i,j}\fs_{i,j}, \ \ 1\leq i\leq m, 1\leq j\leq n_i.
\]
Call such a map a {\bf gauge endomorphism}. When all the $z_{i,j}$ are in $\mb T$ these are the well known {\bf gauge automorphisms}. Let $\gamma_z$ denote the gauge automorphism where every $z_{i,j} = z \in \mb T$. The following is a Fourier series theory for these algebras that is only slightly modified from the work of Davidson and Katsoulis in \cite{DavKat2}.
The map $\Phi_k(a) = \int_{\mb T} \gamma_z(a) \overline{z}^k dz$ is a completely contractive projection of $*_{i=1}^m \fA_{n_i}$ onto the span of all words in the $\fs_{i,j}, 1\leq i\leq m, 1\leq j\leq n_i$, of length $k$. From this we can define a Fourier-like series for our algebra. Define the $k$th Cesaro mean by
\[
\sum_k(a)  = \sum_{i=0}^k (1 - \frac{i}{k}) \Phi_i(a) = \int_{\mb T} \gamma_z(a)\sigma_k(z) dz,
\]
where $\sigma_k$ is the Fejer kernel. Hence, a slight modification of the Fejer Theorem in Fourier analysis gives that 
\[
a = \lim_{k\rightarrow \infty} \sum_k(a) = a, \ \ \forall a\in *_{i=1}^m \cA_{n_i}.
\]
Therefore, for $a\in *_{i=1}^m \fA_{n_i}$ we have that $\Phi_k(a) = 0, \forall k\geq 0$ if and only if $a = 0$.

\begin{lemma}\label{Lemma:injective}
The gauge endomorphisms of $*_{i=1}^m \fA_{n_i}$ are completely contractive and injective.
\end{lemma}
\begin{proof}
By universality $\gamma_{\bf z}$ is a completely contractive homomorphism. For every $k\geq 0$ we have that $\gamma_{\bf z}$ leaves the subspace spanned by words of length $k$ invariant and moreover, is injective on this subspace. Hence, $\gamma_{\bf z}\circ\Phi_k = \Phi_k\circ\gamma_{\bf z}$.

Now, if $a\in *_{i=1}^m \fA_{n_i}$ and $a\neq 0$ then there exists $k\geq 0$ such that $\Phi_k(a) \neq 0$. This implies that $\Phi_k(\gamma_{\bf z}(a)) = \gamma_{\bf z}(\Phi_k(a)) \neq 0$ and so $\gamma_{\bf z}(a) \neq 0$, either. Therefore, $\gamma_{\bf z}$ is injective.
\end{proof}

We are now in a position to prove the analogous result to that in \cite{DavPitts, Pop3}.

\begin{theorem}\label{Thm:Autos}
$\Aut_{ci}(*_{i=1}^m \fA_{n_i}) \simeq \Aut(\times_{i=1}^m \mb B_{n_i})$.
\end{theorem}
\begin{proof}
Let $\varphi \in \Aut_{ci}(*_{i=1}^m \fA_{n_i})$ such that $\varphi^* = \id$. This implies that $\varphi(s_{i,j}) = \fs_{i,j} + C_{i,j}$ where $C_{i,j} \in \ker \pi$ with $\pi$ the canonical quotient map of $*_{i=1}^m \fA_{n_i}$ onto $\cA_{n_1,\cdots, n_m}$. If we can show that $C_{i,j} = 0$ for $1\leq i\leq m, 1\leq j\leq n_i$ then the result is established.

First, consider the completely contractive projection onto the commutative algebra corresponding to the 1st component, $\pi_1 : *_{i=1}^m \fA_{n_i} \rightarrow \cA_{n_1}$, which is obtained by the universal property. Clearly, $\ker \pi \subset \ker \pi_1$. 
Now consider the unital map $\nu : *_{i=1}^m \fA_{n_i} \rightarrow M_2(*_{i=1}^m \fA_{n_i})$ given by
\[
\nu(\fs_{1,j}) = \left[\begin{array}{cc} \fs_{1,j} & 0 \\ 0 & \fs_{1,j}\end{array}\right], \ \ 1\leq j\leq n_1 \ \ {\rm and} 
\]
\[
\nu(\fs_{i,j}) = \left[\begin{array}{cc} 0 & \frac{1}{\sqrt{2}} \fs_{i,j} \\ 0 & \frac{1}{\sqrt{2}} \fs_{i,j} \end{array}\right],\ \  2\leq i \leq n, 1\leq j\leq n_i. 
\]
The row contractive generators are sent to row contractions so by the universal property $\nu$ is a completely contractive homomorphism. 
For example
\[
\nu(\fs_{2,1}\fs_{1,1}\fs_{2,1} + \fs_{1,1}^2) = \left[\begin{array}{ll} \fs_{1,1}^2 & \frac{1}{2}\fs_{2,1}\fs_{1,1}\fs_{2,1} \\ 0 & \frac{1}{2}\fs_{2,1}\fs_{1,1}\fs_{2,1} + \fs_{1,1}^2\end{array}\right].
\]
Note that the (1,1) entry of $\nu(a)$ is $\pi_1(a)$.
Let $z_1 = (1,  \cdots, 1) \in \overline{\mb D}^{n_1}$, $z_i = (\frac{1}{\sqrt 2}, \cdots, \frac{1}{\sqrt 2}) \in \overline{\mb D}^{n_i}$, $2\leq i\leq m$ and ${\bf z} = (z_1,\cdots, z_m)$. By the above theory, If $C\in \ker \pi$ then $\nu(C) = \left[\begin{array}{cc} 0 & \gamma_{\bf z}(C) \\ 0 & \gamma_{\bf z}(C) \end{array}\right]$. Thus,
\[
\nu(\fs_{1,j} + C_{1,j}) = \left[\begin{array}{cc} \fs_{1,j} & \gamma_{\bf z}(C_{1,j}) \\ 0 & \fs_{1,j} + \gamma_{\bf z}(C_{1,j}) \end{array}\right], 1\leq j \leq n_i.
\]
But, $[\nu(\fs_{1,1} + C_{1,1}), \cdots, \nu(\fs_{1,n_1} + C_{1,n_1})]$ is a row contraction which implies that  \newline $[\fs_{1,1}, \gamma_{\bf z}(C_{1,1}), \cdots, \fs_{1,n_1}, \gamma_{\bf z}(C_{1,n_1})]$ is also a row contraction. However, we had assumed that $\fs_{1,1}\fs_{1,1}^* + \cdots + \fs_{1,n_1}\fs_{1,n_1}^* = I$ and so $\gamma_{\bf z}(C_{1,j}) = 0, 1\leq j\leq n_1$. Finally, by Lemma \ref{Lemma:injective} $\gamma_{\bf z}$ is injective and so $C_{1,1} = \cdots C_{1,n_1} = 0$.

Repeating the above argument for $i \in \{2,\cdots, m\}$ yields that $\varphi = \id$.
\end{proof}

\begin{corollary}
If $\varphi : *_{i=1}^{m_1} \fA_{n_i} \rightarrow *_{j=1}^{m_2} \fA_{k_j}$ is a completely isometric isomorphism, then $m_1 = m_2$ and there exists $\alpha \in S_{m_1}$ such that $\varphi = \times_{i=1}^{m_1} \varphi_i$ where $\varphi_i$ is a completely isometric isomorphism from $\fA_{n_i}$ onto $\fA_{k_{\alpha(i)}}, 1\leq i\leq m_1$.
\end{corollary}
\begin{proof}
As before, this isomorphism induces a homeomorphism of the character spaces, $\times_{i=1}^{m_1} \overline{\mb B}_{n_i}$ and $\times_{j=1}^{m_2} \overline{\mb B}_{k_j}$, that is a biholomorphism on their interiors. By \cite{Lig, Tsyganov}, this implies that $m_1 = m_2$ and there exists a permutation $\alpha \in S_{m_1}$ such that $n_{i} = k_{\alpha(i)}, 1\leq i\leq m_1$. As in Lemma \ref{Lemma:surjective} we can define a completely isometric isomorphism $\varphi_\alpha$ that encodes this permutation. 
Therefore,  $\varphi_\alpha^{-1}\circ\varphi\in \Aut_{ci}(*_{i=1}^m \fA_{n_i})$ and the conclusion follows.
\end{proof}


\section{Partition conjugacy}

Consider the following form of conjugacy for multivariable dynamical systems.

\begin{definition}
Two dynamical systems, $(X,\sigma)$ and $(Y,\tau)$ are said to be {\bf partition conjugate} if there exists a homeomorphism $\gamma: X\rightarrow Y$ and clopen sets $V_{i,j} \subset X, 1\leq i, j \leq n$ such that 
\begin{enumerate}
\item $\cup_{i=1}^n V_{i,j} = X$ and $V_{i,j} \cap V_{i', j} = \emptyset $ when $ i\neq i'$.
\item $\cup_{j=1}^n V_{i,j} = X$ and $V_{i,j} \cap V_{i, j'} = \emptyset $ when $ j\neq j'$.
\item $\sigma_i|_{V_{i,j}} = \gamma^{-1} \circ \tau_j \circ \gamma|_{V_{i,j}}$, $1\leq i,j\leq n$.
\item $\sigma_i^{-1}(\sigma_i(V_{i,j})) = V_{i,j} = \gamma^{-1}(\tau_j^{-1}(\tau_j\circ\gamma(V_{i,j}))).$
\end{enumerate}
\end{definition}

One should note that when $X$ is connected this is just conjugacy. In the general setting this is stronger than piecewise conjugacy and weaker than conjugacy. An equivalent condition to (4) is that $\sigma_i(V_{i,j}) \cap \sigma_i(V_{i,j'}) = \emptyset$ for $j\neq j'$ since if $x$ is in this intersection then $\sigma_i^{-1}(x) \subset V_{i,j} \cap V_{i,j'}$, a contradiction.

\begin{theorem}\label{Thm:conjugacy}
If $(X,\sigma)$ and $(Y,\tau)$ are partition conjugate then $C_0(X) \times_\sigma \mb F_n^+$ and $C_0(Y)\times \mb F_n^+$ are completely isometrically isomorphic.
\end{theorem}
\begin{proof}
Let $\gamma, \{V_{i,j}\}_{1\leq i,j\leq n}$ be the given partition conjugacy of the two systems.
Suppose $\mathfrak s_1, \cdots, \mathfrak s_n$ and $\mathfrak t_1, \cdots, \mathfrak t_n$ are the generators of $C_0(X) \times_\sigma \mb F_n^+$ and $C_0(Y) \times_\tau \mb F_n^+$ respectively. Define a covariant representation of $(X,\sigma)$ by
\[
\varphi(f) = f \circ \gamma^{-1}
\]
\[
\varphi(\mathfrak s_i) = \sum_{j=1}^n \mathfrak t_{j} \chi_{\gamma(V_{i,j})}.
\]
From the fact that
\[
\chi_{\tau_j\circ\gamma(V_{i,j})} \mathfrak t_j = \mathfrak t_j \chi_{\tau_j\circ\gamma(V_{i,j})} \circ \tau_j  =  \mathfrak t_j \chi_{\tau_j^{-1}\circ\tau_j\circ\gamma(V_{i,j})}  =  \mathfrak t_j \chi_{\gamma(V_{i,j})}
\]
we can get that for $j\neq j'$
\[
(\mathfrak t_j\chi_{\gamma(V_{i,j})})^* \mathfrak t_{j'} \chi_{\gamma(V_{i,j'})} \ \ = \ \ \chi_{\gamma(V_{i,j})}\mathfrak t_{j}^*\mathfrak t_{j'} \chi_{\gamma(V_{i,j'})}
\]
\[
 = \ \ \mathfrak t_j^*\chi_{\tau_j\circ\gamma(V_{i,j})} \chi_{\tau_{j'}\circ\gamma(V_{i,j'})} \mathfrak t_{j'} \ \ = \ \  
 \mathfrak t_j^*\chi_{\gamma^{-1}\circ\sigma_i(V_{i,j})} \chi_{\gamma^{-1}\circ\sigma_i(V_{i,j'})}\mathfrak t_{j'}  \ \ = \ \ 0.
\]
Thus, $\varphi(\mathfrak s_i)$ is an isometry,
\[
\varphi(\mathfrak s_i)^*\varphi(\mathfrak s_i) = {\left(\sum_{j=1}^n \mathfrak t_{j} \chi_{\gamma(V_{i,j})}\right)}^*\sum_{j=1}^n \mathfrak t_{j} \chi_{\gamma(V_{i,j})}
= \sum_{j=1}^n \chi_{\gamma(V_{i,j})}\mathfrak t_j^*\mathfrak t_j \chi_{\gamma(V_{i,j})} = \sum_{j=1}^n \chi_{\gamma(V_{i,j})} = I.
\]

Additionally, $\varphi$ satisfies the covariance relations.
\[
\varphi(f)\varphi(\mathfrak s_i) = (f\circ\gamma^{-1}) \sum_{j=1}^n \mathfrak t_j \chi_{\gamma(V_{i,j})}
\]
\[
= \sum_{j=1}^n \mathfrak t_j \chi_{\gamma(V_{i,j})} (f\circ\gamma^{-1} \circ\tau_j)
\]
\[
= \sum_{j=1}^n \mathfrak t_j \chi_{\gamma(V_{i,j})} (f\circ\sigma_i\circ\gamma^{-1}) = \varphi(\mathfrak s_i)\varphi(f\circ\sigma_i).
\]
Therefore, by the universal property $\varphi$ extends to a completely contractive homomorphism of $C_0(X) \times_\sigma \mb F_n^+$ into $C_0(Y) \times_\tau \mb F_n^+$.

Similarly, we can define another map $\theta : C_0(Y) \times_\tau \mb F_n^+ \rightarrow C_0(X) \times_\sigma \mb F_n^+$ by
\[
\theta(f) = f\circ\gamma
\]
\[
\theta(\mathfrak t_j) = \sum_{i=1}^n \mathfrak s_i \chi_{V_{i,j}}.
\]
In exactly the same way $\theta(\mathfrak t_j)$ is an isometry which satisfies the covariance relations. Hence, $\theta$ is a completely contractive homomorphism. Therefore, $\theta = \varphi^{-1}$ because it is the inverse on the generators and so $\varphi$ is a completely isometric isomorphism.
\end{proof}

For the $n=2$ discrete case we can specify exactly what this partition conjugacy looks like by giving a partition of the vertices in $X$ where partition conjugacy implies conjugacy on each equivalence class. 

Let $(X,\sigma)$ be a countable discrete dynamical system and $x\in X$. Define the equivalence class of $x$ by putting $z\in [x]$ if $\sigma_i(x) = \sigma_j(z)$, $i,j \in \{1,2\}$ and extending transitively. That is, $[x]$ is the smallest set containing $x$ that satisfies 
\[
[x] = \cup_{1\leq i, j \leq 2} \sigma_i^{-1}(\sigma_j([x])).
\]
Note that the following proposition will not be true if there are three or more maps in the dynamical system.

\begin{proposition}
Let $(X,\{\sigma_1,\sigma_2\})$ and $(Y,\{\tau_1,\tau_2\})$ be partition conjugate countable discrete dynamical systems.
If $X = [x]$, a single equivalence class, then the systems are conjugate.
\end{proposition}
\begin{proof}
Without loss of generality assume that $X=Y$. 
Assume that the systems are not conjugate. Thus, there exists $x_1, x_2\in X$ such that 
\[
\sigma_i(x_1) = \tau_i(x_1), i=1,2,
\]
\[ 
\sigma_1(x_2) = \tau_2(x_2), \sigma_2(x_2) = \tau_1(x_2)
\]
\[
\sigma_1(x_i) \neq \sigma_2(x_i), i = 1,2.
\]
However, $x_1,x_2 \in [x]$ which implies that there exist $i_1,\cdots, i_{2m} \in \{1,2\}$ and points $z_0,\cdots, z_m \in [x] = X$ such that
\[
z_0 = x_1, z_m = x_2
\]
\[
\sigma_{i_{2j-1}}(z_{j-1}) = \sigma_{i_{2j}}(z_{j}), 1\leq j\leq m.
\]
We may also assume that the $z_i$ are all distinct and that the $\sigma_{i_{2j}}(z_j)$ are distinct as well so that the connection between $x_1$ and $x_2$ has no redundancy. This implies that $i_{2j} \neq i_{2j+1}, 1\leq j\leq m-1$.
For ease of notation, let $\sigma_{i_1}(z_0) = z$ and $i_0, i_{2m+1} \in \{1,2\}$ such that $i_0\neq i_1$ and $i_{2m} \neq i_{2m+1}$. Now,
\[
z_0 \in \sigma_{i_1}^{-1}(\{z\}), \ \ z_0 \notin \sigma_{i_0}^{-1}(\{z\}), \ \ z_1\in \sigma_{i_2}^{-1}(\{z\}), \ \ {\rm and} \ \ z_1 \notin \sigma_{i_3}^{-1}(\{z\}) 
\]
There are two cases. First, when $i_1 = i_2$ we have $\{z_0,z_1\} \in \sigma_{i_1}^{-1}(\{z\})$ and not in $\sigma_{i_0}^{-1}(\{z\})$ which implies that $\{z_0,z_1\} \in \tau_{i_1}^{-1}(\{z\})$ by the pre-image condition since $\tau_{i_1}(x_1) = z$. Thus, $\tau_{i_1}(z_1) = \sigma_{i_1}(z_1)$ and $\tau_{i_0}(z_1) = \sigma_{i_0}(z_1)$.

The second case is when $i_1 \neq i_2$ and so is equal to $i_3$. This implies that $z_1 \notin \sigma_{i_1}^{-1}(\{z\})$. Hence, by the pre-image condition $z_1 \notin \tau_{i_1}^{-1}(\{z\})$ since $z_0 \notin \sigma_{i_2}^{-1}(\{z\})$. Hence, $\tau_{i_1}(z_1) = \tau_{i_3}(z_1) = \sigma_{i_1}(z_1)$ and $\tau_{i_0}(z_1) = \sigma_{i_0}(z_1)$. Either case the conclusion is the same.

For the rest of the proof repeat this argument to get that $\tau_{i_1}(z_j) = \sigma_{i_1}(z_j)$ and $\tau_{i_0}(z_j) = \sigma_{i_0}(z_j)$ for $2\leq j\leq m$. This will contradict the assumption that $\sigma_1(x_2) = \tau_2(x_2), \sigma_2(x_2) = \tau_1(x_2)$ and $\sigma_1(x_2) \neq \sigma_2(x_2)$.
\end{proof}


\section{The characterization theorem}

Suppose $\Theta$ is a completely isometric isomorphism of semicrossed product algebras. This will be shown to induce a c.i. isomorphism of some canonical quotient algebras which in turn induces a c.i. isomorphism of free products of noncommutative disc algebras, thus reducing to the theory of Section 2. It will be shown that this rigidity passes back through to the semicrossed product algebras.

\begin{remark}\label{iandci}
Suppose that $\Theta : C_0(X) \times_\sigma \mb F_n^+ \rightarrow \cB$ is a contractive homomorphism. But then $\Theta(C_0(X))$ and $\Theta(\fs_1),\cdots, \Theta(\fs_n)$ form a covariant representation of $(X,\sigma)$ since the generators remain contractions. Therefore, by the universal property $\Theta$ is actually completely contractive. In this and the next sections everything will be stated as being c.c or completely isometric but one should remember that it does not weaken the conclusions to say contractive or isometric.
\end{remark}

\begin{definition}
Given a multivariable dynamical system $(X,\sigma)$ and and a closed subset $X'\subseteq X$, the {\bf sub-dynamical system} $(X',\sigma)$ is the locally compact Hausdorf space $X'$ along with $n$ partially defined proper continuous functions $\sigma_1,\cdots, \sigma_n$ where $\sigma_i$ is defined at $x\in X'$ if and only if $\sigma_i(x) \in X'$.
\end{definition}

One can develop this in general but we are interested only in the case of a finite number of points, so assume $|X'| < \infty$.
For such a closed subset $X'$ of $X$ consider the closed ideal $J_{(X',\sigma)}$ of $C_0(X) \times_\sigma \mb F_n^+$ generated by $f\in C_0(X)$ such that $f(x) = 0, \forall x\in X'$.
Define 
\[
\pi_{(X',\sigma)} : C_0(X) \times_\sigma \mb F_n^+ \rightarrow (C_0(X) \times_\sigma \mb F_n^+)/J_{(X',\sigma)}
\] 
to be the canonical quotient map. This quotient algebra is generated by $C(X')$ and $\pi_{(X',\sigma)}(\fs_i), 1\leq i\leq n$ so it can be concretely represented as a subalgebra of $M_{|X'|}(B(\cH)) = B(\cH^{(|X'|)})$ where we will label $\cH^{(|X'|)} = \oplus_{x\in X'} \cH_x$. Define $P_x$ to be the projection onto $\cH_x$ for every $x\in X'$ and 
\[
\theta_{x,y} : B(\cH^{(|X'|)}) \rightarrow B(\cH_x, \cH_y) = B(\cH)
\]
to be the compression map onto the $(y,x)$ entry of each matrix. Of course, such a $\theta_{x,y}$ is linear and contractive. As well, this satisfies 
\[
\theta_{x,y}(P_y A P_x) = \theta_{x,y}(A), \ \forall A\in B(\cH^{(|X'|)}).
\]
Such a map becomes useful since $\theta_{x,y}(\pi_{(X', \sigma)}(\fs_i)) \neq 0$ if and only if $\sigma_i(x) = y$, for any $x,y\in X'$.

First, we want to show that this quotient has some form of canonicity. To achieve this we turn to operator algebras of finite directed graphs. Let $G = (V, E, r, s)$ be a finite directed graph with $V$ the vertex set, $E$ the edge set, and $r$ and $s$ the range and source maps. The {\bf Cuntz-Krieger algebra} of $G$, $\cO(G)$, is the universal C$^*$-algebra generated by Cuntz-Krieger families $\{S,P\}$ on $\cH$ where this family $\{S,P\}$ consists of a set $\{P_v :v\in V\}$ of mutually orthogonal projections on $\cH$ and a set $\{S_e : e\in E\}$ of partial isometries on $\cH$ such that
\\ \indent $\bullet$ ${S_e}^* S_e = P_{s(e)}$ for all $e\in E$,
\\ \indent $\bullet$ ${S_e}^* S_f = 0$ if $e\neq f$, and 
\\ \indent $\bullet$ $P_v = \sum_{\{e\in E: r(e) = v\}} S_e{S_e}^*$ whenever $v$ is not a source.

The norm closed subalgebra of $\cO(G)$ generated by the $\{S_e\}$ and the $\{P_v\}$ is called the {\bf tensor algebra} and denoted $\cT_+(G)$.  The tensor algebra was first introduced by Muhly and Solel in \cite{MuhlySolel} under the name {\em quiver algebra} in the context of C$^*$-correspondences and Toeplitz algebras and was shown to be equivalent to the above definition in \cite{FowlerRaeburn}. A good review of all things graph theory can be found in \cite{Raeburn}.

 For $G = (V, E, r,s)$ with $E = E_1\dot\cup \cdots \dot\cup E_n$ Duncan \cite{Duncan1} defines an edge-colored Cuntz-Krieger family $\{S,P\}$ on $\cH$ where $\{\{S_e : e\in E_i\}, P\}$ is a Cuntz-Krieger family for $(V, E_i, r,s)$ for $1\leq i\leq n$. 
 
 To every finite sub-dynamical system, $(X', \sigma)$, one can associate a graph. In particular, let $V= X'$ and whenever $x, y\in X'$ with $\sigma_i(x) = y$  define $e\in E_i$ with $s(e) = x$ and $r(e) = y$. Taking $E = \dot\cup_{i=1}^n E_i$ we see that $G = (V,E)$ is a finite directed graph with a finite number of edges. For each $1\leq i\leq n$ let $G_i = (V,E_i)$.

\begin{lemma}
For a sub-dynamical system $(X', \sigma)$ of $(X,\sigma)$ with $|X'| < \infty$, the quotient $(C_0(X) \times_\sigma \mb F_n^+)/J_{(X',\sigma)}$ is completely isometrically isomorphic to the free product $*_{i=1}^n \cT_+(G_i)$ amalgamated over $\spn\{P_v : v\in V\} \simeq C(X')$.
\end{lemma}
\begin{proof}
Recall that we are considering this quotient as a subalgebra of $M_{|X'|}(B(\cH))$ and that $P = \{P_x : x\in V = X'\}$ is a set of mutually orthogonal projections. For each $e\in E_i$, set $S_e = P_{r(e)}\pi_{(X',\sigma)}(\fs_i)P_{s(e)}$ noting that $\theta_{s(e),r(e)}(S_e)$ is a contraction. Now, for $y\in X'$ and  $1\leq i\leq n$ let $\{e_1,\cdots, e_k\} = r^{-1}(\{y\}) \cap E_i$ we have that $P_y\pi_{(X',\sigma)}(S_i)(\sum_{j=1}^k P_{s(e_j)})$ is a contraction and thus
\[
[\theta_{s(e_1), y}(\pi_{(X',\sigma)}(\fs_i)), \  \cdots\ , \ \theta_{s(e_k), y}(\pi_{(X',\sigma)}(\fs_i))]
\]
is a row contraction. It should be noted that for $\pi_{(X',\sigma)}(\fs_i)$ there is at most one entry in each column since the $\sigma_i$ are functions.

Considering the $S_e$ as operators on $B(\cH)$ we get a set of row contractions indexed over $r(E) \subseteq X'$. As was mentioned before, Popescu \cite{Pop1} showed that this dilates to a sequence of row isometries $\{V_e : e\in E\}$ on $B(\cK)$ such that $\sum_{e \in E_i, r(e) = x} V_eV_e^* = I$ for each $x\in r(E_i), 1\leq i\leq n$.
Construct operators $R_i \in M_{|X'|}(B(\cK))$ such that the $P_y R_i P_x = V_e$ when $s(e) = x$ and $r(e) = y$ for $e\in E_i$. The operator algebra generated by $C(X')$ and the $R_i$ is generated by an edge-colored Cuntz-Krieger family, a fact that will be used later.

Now, for $1\leq i\leq n$ and $y\in X'$
\[
P_yR_i = R_i\big(\sum_{\sigma_i(x) = y} P_x\big).
\]
Thus for $f\in C_0(X)$ we have that
\[
\pi_{(X',\sigma)}(f) R_i = R_i \pi_{(X',\sigma)}(f\circ\sigma_i).
\]
Thus, $\pi_{(X',\sigma)}$ and $R_1,\cdots, R_n$ form a covariant representation of $(X,\sigma)$. Hence, by the universal property there exists $\theta : C_0(X) \times_\sigma \mb F_n^+$ onto $\overline\alg\{C(X'), R_1,\cdots, R_n\}$ such that $\theta(\fs_i) = R_i, 1\leq i\leq n$. Furthermore, for every $f\in C_0(X)$ such that $f|_{X'} = 0$, $\theta(f) = \pi_{(X',\sigma)}(f) = 0$ and so $\theta(J_{(X',\sigma)}) = 0$.
In other words, $(C_0(X) \times_\sigma \mb F_n^+)/J_{(X',\sigma)}$ is completely isometrically isomorphic to $\overline\alg\{C(X'), R_1,\cdots, R_n\}$ and the conclusion follows.
\end{proof}

\begin{definition}
Let $\cB \subseteq M_k(B(\cH))$ be an operator algebra. The {\bf entry algebra} of $\cB$ is the operator algebra $E(\cB) \subset B(H)$ generated by $B_{i,j}, 1\leq i,j \leq k$, where $B = [B_{i,j}]_{i,j=1}^k$ for every $B\in\cB$.
\end{definition}

Applying this to our context we recover a familiar object.

\begin{corollary}\label{Cor:freeproduct}
If $X' = \{x_1,\cdots, x_k\} \subset X$ then $E((C_0(X) \times_\sigma \mb F_n^+)/J_{(X',\sigma)})$ is equal to $*_{i=1}^m \fA_{n_i}$.
\end{corollary}
\begin{proof}
From the proof of the previous proposition we see that 
\[
\{\theta_{x,y}(\pi_{(X',\sigma)}(\fs_i)) : x\in X', \sigma_i(x) = y\}, \ \ y\in X', \ 1\leq i\leq n
\]
is a set of row isometries. By the proof of the previous proposition these row isometries are free from each other.
\end{proof}

In all of the following theory we take $X=Y$ and assume that any completely isometric isomorphism, $\Theta$, is the identity on $C_0(X)$. This we can do without loss of generality because $\Theta$ maps $C_0(X)$ to $C_0(Y)$ $*$-isomorphically and $(Y,\tau)$ is conjugate to $(\gamma(Y), \gamma\circ\tau\circ\gamma^{-1})$ for any homeomorphism $\gamma$ implying that the semicrossed products of those two systems are completely isometrically isomorphic.

\begin{lemma}\label{Lemma:entryalgebra}
If $\Theta : C_0(X) \times_\sigma \mb F_n^+ \rightarrow C_0(X) \times_\tau \mb F_n^+$ is a completely isometric isomorphism and $X' \subseteq X$ with $|X'| < \infty$ then $\Theta$ induces a completely isometric isomorphism $\varphi_{X'}$ of the entry algebras $E((C_0(X)\times_\sigma \mb F_n^+)/J_{(X', \sigma)})$ and $E((C_0(X) \times_\tau \mb F_n^+)/J_{(X', \tau)})$.
\end{lemma}
\begin{proof}
Such a completely isometric isomorphism $\Theta$, because it fixes $C_0(X)$, takes $J_{(X,\sigma)}$ isomorphically onto $J_{(X',\tau)}$. Hence, it induces a completely isometric isomorphism,  $\Theta_{X'}$,  between the quotients such that $\pi_{(X',\tau)}\circ \Theta = \Theta_{X'} \circ \pi_{(X',\sigma)}$. This implies that  $\Theta_{X'}|_{C(X')} = \id$ and so
\[
\Theta_{X'}(P_y\pi_{(X',\sigma)}(\fs_i)P_x)  =  P_y \pi_{(X', \tau)}(\Theta(\fs_i))P_x  \in P_y M_{|X'|}(E((C_0(X)\times_\tau \mb F_n^+)/J_{(X',\tau)}))P_x
\]
provides the required map on the generators of the entry algebra. In other words, define $\varphi_{X'}$ on the generators of the entry algebra by
\[
\varphi_{X'}(\theta_{x,y}(\pi_{(X',\sigma)}(\fs_i))) \ \ = \ \ \theta_{x,y}(\pi_{(X',\tau)}(\Theta(\fs_i))).
\]
$\Theta(\fs_i)$ is a contraction so $\big(\theta_{y,x}(\pi_{(X',\tau)}(\Theta(\fs_i)))\big)_{x\in X'}$ is a row contraction.
Recall that Davidson and Katsoulis showed that completely isometrically isomorphic algebras implied piecewise conjugate dynamical systems \cite{DavKat2} which shows that $\varphi_{X'}$ is non-zero on the generators.
This map is well defined since the structure of the entry algebras was shown to be a free product in the previous corollary. By the universal property of the free product of noncommutative disc algebras $\varphi_{X'}$ is a completely contractive homomorphism.  In the same way, $\Theta_{X'}^{-1}$ induces a completely contractive homomorphism from $E((C_0(X)\times_\tau \mb F_n^+)/J_{(X',\tau)})$ to $E((C_0(X)\times_\sigma \mb F_n^+)/J_{(X',\sigma)})$ that is the inverse on the generators. Therefore, $\varphi_{X'}$ is a completely isometric isomorphism.
\end{proof}

The existence of this entry algebra isomorphism is entirely dependent on the form of these algebras. For instance, if $U_1$ and $U_2$ are the generators of $A(\mb D) * A(\mb D)$ then the algebra generated by $U_1\oplus U_1$ and the algebra generated by $U_1\oplus U_2$ are both completely isometrically isomorphic to $A(\mb D)$ but their entry algebras are $A(\mb D)$ and $A(\mb D)*A(\mb D)$ respectively, which we know to be not isometrically isomorphic. Remarkably this is even false for C$^*$-algebras, Plastiras \cite{Plastiras} produces two non-isomorphic C$^*$-algebras that are isomorphic after tensoring with $M_2$.

\begin{proposition}\label{Prop:permutation}
For a completely isometric isomorphism $\Theta : C_0(X) \times_\sigma \mb F_n^+ \rightarrow C_0(X) \times_\tau \mb F_n^+$ and $x\in X$ there exists a unique permutation $\alpha_{\Theta, x} \in S_n$ such that for every finite $X' \subset X$ with $x\in X'$ we have
\[
\pi_{(X', \tau)}(\Theta(\fs_i))P_x \in \overline{\alg}\{C(X'), \pi_{(X',\tau)}(\ft_{\alpha_{\Theta,x}(i)})\}P_x, \ \ 1\leq i\leq n.
\]
\end{proposition}
\begin{proof}
Fix $x\in X$. First consider $X' = \{x, \sigma_1(x), \cdots, \sigma_n(x)\} = \{x, \tau_1(x), \cdots, \tau_n(x)\}$, because $(X,\sigma)$ and $(X,\tau)$ are piecewise conjugate.
Lemma \ref{Lemma:entryalgebra} shows that $\Theta$ induces a completely isometric isomorphism $\varphi_{X'} : E((C_0(X)\times_\sigma \mb F_n^+)/J_{(X',\sigma)}) \rightarrow E((C_0(X)\times_\tau \mb F_n^+)/J_{(X',\tau)})$. Corollary \ref{Cor:freeproduct} gives that $\varphi_{X'}$ is a completely isometric isomorphism of free products of noncommutative disc algebras. By the corollary to Theorem \ref{Thm:Autos} this implies that both entry algebras are made up of the same noncommutative disc algebras and after rearrangement $\varphi_{X'}$ is just a product of noncommutative disc algebra completely isometric automorphisms.

Therefore, there exists a permutation $\alpha\in S_{n\times|X'|}$ such that for each $1\leq i, j\leq n$, the row contraction $\big( \theta_{y,\sigma_j(x)}(\pi_{(X',\sigma)}(\fs_i))\big)_{y\in X'}$ is mapped into the noncommutative disc algebra generated by the row contraction $\big( \theta_{y,\tau_{j'}(x)}(\pi_{(X',\tau)}(\ft_{i'}))\big)_{y\in X'}$ where $\alpha(i,\sigma_i(x)) = (i', \tau_{j'}(x))$.

Define the required permutation $\alpha_{\Theta,x} \in S_n$ by $\alpha(i, \sigma_i(x)) = (\alpha_{\Theta,x}(i), \tau_j(x)), 1\leq i\leq n$. Thus, if $E_{y,z}$ is the matrix unit in $B(\cH^{|X'|})$ taking $\cH_z$ onto $\cH_y$, then
\[
\pi_{(X',\tau)}(\Theta(\fs_i))P_x = \Theta_{X'}(\pi_{(X',\sigma}(\fs_i))P_x 
\]
\[
= \Theta_{X'}(P_{\sigma_i(x)}\pi_{(X',\sigma)}(\fs_i)P_x)
= P_{\sigma_i(x)}\pi_{(X',\tau)}(\Theta(\fs_i))P_x   
\]
\[= \theta_{x, \sigma_i(x)}(\pi_{(X',\tau)}(\Theta(\fs_i)) E_{\sigma_i(x),x}
= \varphi_{X'}(\theta_{x, \sigma_i(x)}(\pi_{(X',\sigma)}(\fs_i)))E_{\sigma_i(x),x} 
\]
\[\in \overline\alg\{C(X'), \theta_{x, \tau_j(x)}(\pi_{(X',\tau)}(\ft_{\alpha_{\Theta,x}(i)}))E_{\tau_j(x),x}, 1\leq j\leq n\}P_x
\]
\[
= \overline\alg\{C(X'), \pi_{(X',\tau)}(\ft_{\alpha_{\Theta,x}(i)})P_x.
\]

Finally, we need to establish that this permutation is unique. Let $X'' \subset X$ be any finite set such that $x\in X''$ and let $\tilde X = X'' \cup X'$.
Again, because $\Theta$ is the identity on $C_0(X)$
\[
\pi_{(\tilde X, \tau)}(\Theta(\fs_i))P_x = \theta_{x,\sigma_i(x)}(\pi_{(\tilde X, \tau)}(\Theta(\fs_i)))E_{\sigma_i(x),x}
\]
\[
= \varphi_{\tilde X}(\theta_{x,\sigma_i(x)}(\pi_{(\tilde X,\sigma)}(\fs_i)))E_{\sigma_i(x),x}
\]
\[
\in \overline\alg\{ C(\tilde X), \theta_{x,y}(\pi_{(\tilde X,\tau)}(\ft_{j}))E_{y,x},  \forall y\in \tilde X\}P_x
\]
\[
= \overline\alg\{ C(\tilde X), \pi_{(\tilde X, \tau)}(\ft_j)\}P_x
\]
for some $1\leq j\leq n$. However, $P_{X'}(\pi_{(\tilde X,\sigma)})P_{X'}$ is identical to $\pi_{(X', \sigma)}$ (which is true as well when $\sigma$ is replaced by $\tau$) and so $j$ must be equal to $\alpha_{\Theta,x}(i)$ by the first argument. Finally, compression to $\cH^{(|X''|)} \subset \cH^{(|\tilde X|)}$ gives the desired conclusion.
\end{proof}

Out of the proof of the previous proposition we get what will become the pre-image condition:

\begin{corollary}\label{Cor:pre-image}
For $x,y\in X$ if $\sigma_i(x) = \sigma_i(y)$ then $\alpha_{\Theta,x}(i) = \alpha_{\Theta, y}(i)$  and \\ if
$\tau_j(x) = \tau_j(y)$ then $\alpha^{-1}_{\Theta,x}(j) = \alpha^{-1}_{\Theta, y}(j)$.
\end{corollary}
\begin{proof}
If $z= \sigma_i(x) = \sigma_i(y)$ then for $X' = \{x,y,z\} \subset X$ we have that 
\[
(\phi_{x,z}(\pi_{(X',\sigma)}(\fs_i)), \ \phi_{y,z}(\pi_{(X',\sigma)}(\fs_i))
\]
is a row contraction of elements in $E((C_0(X) \times_\sigma \mb F_n^+)/J_{(X',\sigma)})$. As such $\varphi_{X'}$ will map this into the same noncommutative disc algebra generated by the row contraction 
\[
(\phi_{x,z}(\pi_{(X',\tau)}(\ft_j)), \phi_{y,z}(\pi_{(X',\tau)}(\ft_j)), \phi_{z,z}(\pi_{(X',\tau)}(\ft_j))) 
\]
for some $1\leq j\leq n$. The previous proposition dictates that $\alpha_{\Theta,x}(i) = j = \alpha_{\Theta, y}(i)$.

The last conclusion follows from the fact that $\alpha_{\Theta^{-1},x} = \alpha_{\Theta,x}^{-1}$.
\end{proof}

We now will show that these permutations $\alpha_{\Theta,x}$ change continuously with respect to $x$.

\begin{proposition}\label{Prop:continuous}
Let $\alpha_\Theta : X \rightarrow S_n$ be defined as $\alpha_\Theta(x) = \alpha_{\Theta,x}$.
Then $\alpha_\Theta$ is a continuous function where $S_n$ has the discrete topology.
\end{proposition}
\begin{proof}
Let $\{x_k\} \in X$ be a sequence converging to $x\in X$ such that $\alpha_\Theta(x_k) = \alpha \in S_n, \forall k\geq 1$. Assume that $\alpha_{\Theta, x} \neq \alpha$ and so there exists $1\leq i\leq n$ such that $\alpha_{\Theta, x}(i) \neq \alpha(i)$.

In the same way as in the previous section we can define the gauge automorphisms of $C_0(X) \times_\tau \mb F_n^+$ to be $\gamma_z, z = (z_1,\cdots, z_n) \in \mb T^n$ by letting $\gamma_z(\fs_i) = z_iS_i$ and extending to a completely isometric isomorphism by the universal property.
Again, define for $z = (z,\cdots, z) \in \mb T^n$
\[
\Phi_k(a) \ \ = \ \ \int_{\mb T} \gamma_z(a) \overline{z}^k dz,
\]
the completely contractive projection onto words of length $k$ of the generators $S_1,\cdots, S_n$. Finally, the sum of these projections converges to $a$ by way of the Cesaro means, giving us that $\Phi_k(a) = 0, \forall k\geq 1$ if and only if $a = 0$.

Now, look at $\Theta(\fs_i)$ where we have $\alpha_{\Theta, x}(i) \neq \alpha(i)$. $\Theta(\fs_i) \notin C_0(X)$ since $\Theta^{-1}|_{C_0(X)} = \id$ there exists $k\geq 1$ such that 
\[
\Phi_k(\Theta(\fs_i)) = \sum_{w\in \mb F_n^+, |w| \leq k} \fs_w f_w \neq 0.
\] 
Claim: $f_w(x) \neq 0$ if and only if $w = \alpha_{\Theta, x}(i)^k$.
\vskip 6 pt

\noindent Verification: Let $X' = \{x\} \cup \{ \sigma_w(x) : w\in \mb F_n^+, |w| \leq k\} \subseteq X$. 
Hence, 
\[
\pi_{(X',\sigma)}(\fs_wf_w)P_x = \pi_{(X',\sigma)}(\fs_w)f_w(x)P_x = 0 \ \ \  \textrm{if and only if} \ \  \ f_w(x) = 0
\]
because $\pi_{(X',\sigma)}(\fs_w)P_x \neq 0$.
Now, Proposition \ref{Prop:permutation} says that $\pi_{(X', \tau)}(\Theta(\fs_i))P_x\neq 0$ is in the algebra generated by $T_{\alpha_{\Theta,x}(i)}$ times $P_x$. 

Next, we can define the projections $\Phi_k$ for the quotient as well such that $\Phi_k \circ\pi_{(X',\tau)}  = \pi_{(X',\tau)} \circ\Phi_k$. 
Moreover, if $A\in \overline\alg\{C(X'), \pi_{(X',\tau)}(T_j)\}$ then so is $\Phi_k(A)$.
Thus, 
\[
\pi_{(X',\tau)}(\Phi_k(\Theta(\fs_i)))P_x = \Phi_k(\Theta(\pi_{(X',\sigma)}(\fs_i)))P_x \in \overline\alg\{C(X'), \pi_{(X',\tau)}(\ft_{\alpha_{\Theta,x}(i)})\}P_x.
\]
The claim follows.

Lastly, the $f_w$'s are continuous functions so there is a $m\geq 1$ such that $f_w(x_m) \neq 0$ for $w = \alpha_{\Theta, x}(i)^k$ as well. The previous claim is still true when we replace $x$ with $x_m$. However, this implies that $\alpha = \alpha_{\Theta, x_m}(i) = \alpha_{\Theta, x}(i)$, a contradiction. Therefore, $\alpha_\Theta$ is a continuous function.
\end{proof}

Finally, the second main result of the paper, that the semicrossed product algebra completely characterizes multivariable dynamical systems up to partition conjugacy.

\begin{theorem}
The dynamical systems $(X,\sigma)$ and $(Y,\tau)$ are partition conjugate if and only if $C_0(X) \times_\sigma \mb F_n^+$ and $C_0(Y) \times_\tau \mb F_n^+$ are completely isometrically isomorphic.
\end{theorem}
\begin{proof}
Theorem \ref{Thm:conjugacy} proves one direction. So assume that $\Theta : C_0(X) \times_\sigma \mb F_n^+ \rightarrow C_0(Y) \times_\tau \mb F_n^+$ is a completely isometric isomorphism. Again, without loss of generality we can assume that $X = Y$ and that $\Theta|_{C_0(X)} = id|_{C_0(X)}$.

Define
\[
V_{i,j} \ = \ \{x\in X: \alpha_{\Theta, x}(i) = j\}, \ 1\leq i,j\leq n,
\]
where the $\alpha_{\Theta,x}$ is the canonical permutation obtained from $\Theta$ and $x$ found in Proposition \ref{Prop:permutation}. Since this runs through all of $X$ we have that
\[
\dot\cup_{j=1}^n V_{i,j} = X, 1\leq i\leq n, \ \ {\rm and} \ \ \dot\cup_{i=1}^n V_{i,j}  = X, 1\leq j\leq n.
\]
Furthermore, $\alpha_\Theta^{-1}(\{\gamma\})$ is clopen for every $\gamma \in S_n$ by Proposition \ref{Prop:continuous} and so 
\[
V_{i,j} = \cup_{\gamma(i) = j, \gamma\in S_n} \alpha_\Theta^{-1}(\{\gamma\})
\]
is clopen as well. In the proof of the same proposition it was shown that for every $x\in V_{i,j}$ we have that $\sigma_i(x) = \tau_j(x)$.

Lastly, if $x\in \sigma_i^{-1}(\sigma_i(V_{i,j})$ then there exists $y\in V_{i,j}$ such that $\sigma_i(x) = \sigma_i(y)$. Corollary \ref{Cor:pre-image} then says that $\alpha_{\Theta,x}(i) = \alpha_{\Theta, y}(i) = j$ and so $x\in V_{i,j}$. Repeat the same argument to obtain $\tau_j^{-1}(\tau_j(V_{i,j})) = V_{i,j}$ as well. Therefore, $\{V_{i,j}\}$ is the desired partition and the two dynamical systems are partition conjugate.
\end{proof}


\section{The tensor algebra}

We conclude this paper by showing when the tensor algebra and the semicrossed product algebra of a dynamical system $(X,\sigma)$ are completely isometrically isomorphic, namely if and only if the $\sigma_i$ have pairwise disjoint ranges. This is achieved by the following representation theory.

Introduced independently in \cite{KatsKribs} and \cite{Solel}, nest representations form an important class in the representation theory of non-selfadjoint operator algebras. To be precise a nest representation is a representation whose lattice of invariant subspaces forms a nest, that is, is linearly ordered. For instance, irreducible representations are nest representations 

In \cite{DavKat1} and \cite{DavKat2}, 2-dimensional nest representations with lattice of invariant subspaces equal to $\{\{0\}, \mb Ce_1, \mb C^2\}$ were used to great effect. We generalize this idea in the following manner:

\begin{definition}
Consider the following subalgebra of the upper triangular matrices, $\cT_m$
\[
\cN_m \ \ = \ \ \left\{ \left[\begin{array}{cccc} x_{1,1} & x_{1,2}& \cdots & x_{1,m} \\  & x_{2,2}   \\ &&\ddots \\ &&&x_{m,m} \end{array}\right] : x_{i,i}, x_{1,i} \in \mathbb C, 1\leq i\leq m \right\}.
\]
If $\cA$ is an operator algebra, let $\rep_{\cN_m}(\cA)$ denote the collection of completely contractive homomorphisms $\rho : \cA \rightarrow \cT$ such that compression to the $(i, j)$ entry, call it $\rho_{i,j}$, is a norm 1 linear functional, for all $1\leq i = j\leq m$ and $i=1, 2\leq j\leq m$.
\end{definition}

The avid reader will note that this is no longer a nest representation for $m\geq 3$.
Notice that compression to any of the diagonal entries is a non-trivial homomorphism into $\mb C$, that is, a character in $\cM(\cA)$. Recall that the only characters of $C_0(X)$ are point evaluations. Thus , for $\rho \in \cM(C_0(X) \times_\sigma \mb F_n^+)$, there exists $x\in X$ such that $\rho$ restricted to $C_0(X)$ is point evaluation at $x$. In \cite{DavKat2} it was shown that the set of characters corresponding to point evaluation at $x\in X$ is homeomorphic to $\overline{\mb D}^k$ where $k$ is the number of maps that fix the point $x$.
Hence, for $\rho \in \rep_{\cN_m}(C_0(X) \times_\sigma \mathbb F^n_+)$ there exist $x_1,\cdots, x_m \in X$ such that $\rho_{i,i}(f) = f(x_i), \forall f\in C_0(X)$. Note that $\rho(C_0(X))$ is diagonal because $\rho$ is completely contractive and so a $*$-homomorphism on $C_0(X)$.

For $1\leq i\leq m$ consider the map $\theta_i(A) = \left[\begin{array}{cc} \rho_{1,1}(A) & \rho_{1,i+1}(A) \\ 0 & \rho_{i+1,i+1}(A) \end{array}\right]$. $\theta$ is a nest representation and by \cite[Lemma 3.15]{DavKat2} there exists $1\leq j\leq n$ such that $\sigma_j(x_i) = x_1$. 
In light of this, for ${\bf x} = (x_1,\cdots, x_m)$ denote $\rep_{x, {\bf x}}(C_0(X) \times_\sigma \mb F_n^+)$ to be the subset of $\rep_{\cN_{m+1}}$ such that $\rho_{1,1}$ is a character corresponding to $x$ and for $1\leq i\leq m$, $\rho_{i+1,i+1}$ is a character corresponding to $x_i$. In other words, if $\rep_{x, \bf x}(C_0(X) \times_\sigma \mb F_n^+)$ is non-empty then $\{x_1,\cdots, x_m\} \subset \cup_{i=1}^n \sigma^{-1}_i(\{x\})$.

The following example was inspired by \cite{Duncan2} where they define a similar 3x3 representation to identify when two edges of an edge-colored directed graph belong to different colors.

\begin{example}\label{Example:Colours}
Let $(X,\sigma)$ be a dynamical system with $x, x_1,\cdots, x_m\in X$ for $m\leq n$ and suppose that $\sigma_i(x_i) = x$ for $1\leq i\leq m$. Define a representation $\rho$ by 
\[
\rho(f) = \left[\begin{array}{cccc} f(x) & 0 & \cdots & 0 \\  & f(x_1)  \\  &  & \ddots \\ &&& f(x_m) \end{array}\right] \ \ {\rm and} 
\ \ \rho(\mathfrak s_k) = \left[\begin{array}{cccc} 0 & \delta_{1k} & \cdots & \delta_{mk} \\  &0 \\ &&\ddots \\ &&&0 \end{array}\right]
\]
where $\delta_{ik}$ is the Kronecker delta function. Thus, for $A = \sum_{w\in \mb F_n^+} \mathfrak s_w f_w$ where all but finitely many of the $f_w \in C_0(X)$ are 0
\[
\rho(A)  = \left[\begin{array}{cccc} f_0(x) & f_1(x_1) & \cdots & f_m(x_m) \\  & f_0(x_1) \\ &&\ddots \\ &&&f_0(x_m) \end{array}\right]
\]
is a completely contractive homomorphism onto $\cN_{m+1}$ which clearly satisfies the $\rho_{i,j}$ norm 1 condition. Therefore, $\rho$ extends to a c.c. homomorphism on all of $C_0(X) \times_\sigma \mb F_n^+$ and so, for ${\bf x} = (x_1,\cdots, x_m)$, we have $\rho \in \rep_{x, \bf x}(C_0(X) \times_\sigma \mathbb F^n_+)$.
\end{example}

Although not necessary to prove the main theorem of this section we have the following converse to the previous example.

\begin{proposition}\label{Prop:Colours}
Let $\cA = C_0(X) \times_\sigma \mb F_n^+$, $x,x_1,\cdots, x_m \in X$ and ${\bf x} = (x_1,\cdots, x_m)$. If $\rep_{x,\bf x}(\cA)$ is non-empty then there exist distinct $1\leq i_1, \cdots, i_m \leq n$ such that $\sigma_{i_1}(x_1) = \cdots = \sigma_{i_m}(x_m) = x$.
\end{proposition}
\begin{proof}
Let $\rho \in \rep_{x,\bf x}(\cA)$. For each $1\leq i\leq m$ consider the compression to the 1st and $i+1$st rows and columns
\[
\theta_i = \left[\begin{array}{cc} \rho_{1,1} & \rho_{1,i+1} \\ 0 & \rho_{i+1,i+1}\end{array}\right].
\]
Let $\mathfrak s_{i_1}, \cdots, \mathfrak s_{i_k}$ be those generators not in the kernel of $\rho_{1,i+1}$ and let $t = \max_{j=1}^k \|\theta_i(\mathfrak s_{i_j})\|$.
If $t < 1$ then consider the covariant representation of $(X,\sigma)$ given by $\theta_i(C_0(X))$ and 
\[
T_j = \left[\begin{array}{cc} \rho_{1,1}(\mathfrak s_j) & (2-t)\rho_{1,i+1}(\mathfrak s_j) \\ 0 & \rho_{i+1,i+1}(\mathfrak s_j)\end{array}\right].
\] 
Note that by the triangle inequality this still gives $\|T_j\| \leq 1$. Hence, by the universal property there exists a completely contractive homomorphism $\theta'_i$ of $\cA$ onto $\cT_2$, the upper triangular matrices taking $\mathfrak s_j$ to $T_j$. Moreover, we have that
\[
\theta'_i = \left[\begin{array}{cc} \rho_{1,1} & (2-t)\rho_{1,i+1} \\ 0 & \rho_{i+1,i+1}\end{array}\right].
\]
Now, contractive implies that $|(2-t)\rho_{1,i+1}| \leq 1$ but then $|\rho_{1,i+1}| \leq \frac{1}{2-t} < 1$, a contradiction since it is a norm 1 linear functional.
Hence, there must be at least one $1\leq j\leq k$ such that $\|\theta_i(\mathfrak s_{i_j})\| = 1$ and because $\rho_{1,i}(\mathfrak s_{i_j}) \neq 0$ this implies that $\rho_{1,l}(\mathfrak s_{i_j}) = 0$ for all $l \neq 1, i_j$. 
Therefore, $\sigma_{i_j}(x_i) = x$ and $\theta_k(\mathfrak s_{i_j})$ is diagonal for all $k\neq i$.
\end{proof}

\begin{theorem}
Let $(X,\sigma)$ be a dynamical system, then $\cA(X,\sigma)$ and $C_0(X) \times_\sigma \mathbb F^n_+$ are completely isometrically isomorphic if and only if the ranges of the $\sigma_i$ are pairwise disjoint. \end{theorem}
\begin{proof}
Suppose that there are points $x, x_1, x_2\in X$ such that $\sigma_1(x_1) = \sigma_2(x_2) = x$. If $x_1 = x_2 = x$ then there is a point fixed by more than one map and Davidson and Katsoulis have shown in \cite[Corollary 3.11]{DavKat2} that $\cA(X,\sigma)$ and $C_0(X) \times_\sigma \mb F_n^+$ are not even algebraically isomorphic due to their differing character spaces.

Assume then that $x_1 \neq x$. If ${\bf x} = (x_1, x_2)$ Example \ref{Example:Colours} says that $\rep_{x, \bf x}(C_0(X) \times_\sigma \mb F_n^+)$ is non-empty. Again we turn to \cite{DavKat2} to see that a completely isometric isomorphism $\gamma$ between these algebras induces a bijection $\gamma_c$ between the character spaces that corresponds to the homeomorphism $\gamma_s$ that arises from $\gamma$ being a $*$-automorphism of $C_0(X)$. Without loss of generality we may take $\gamma_s$ to be the identity since conjugacy among dynamical systems implies completely isometrically isomorphic among the tensor algebras. Hence, there must exist a $\rho \in \rep_{x,\bf x}(\cA(X,\sigma))$ when the algebras are c.i.i.

In light of this, consider, as in the proof of Proposition \ref{Prop:Colours}, the nest representation $\theta_1$ given by the compression to the upper left 2x2 corner of $\rho$. Let $i_1,\cdots, i_k$ be the indices such that $\rho_{1,2}(\mathfrak s_{i_j}) \neq 0$, implying that $\sigma_{i_j}(x_1) = x$ and, since $x_1\neq x$, that $\rho_{2,2}(\mathfrak s_{i_j}) = 0$ (since it is a character and $\sigma_{i_j}$ does not fix $x_1$). Let $t = \|[\theta_1(\mathfrak s_{i_1}) \cdots \theta_1(\mathfrak s_{i_k})]\|$. Because the generators of $\cA(X,\sigma)$ are row contractive and $\theta_1$ is c.c. then $t\leq 1$. If $t< 1$ then consider the slight adjustment to $\theta_1$ given by
\[
\theta'_1 = \theta_1 + \left[\begin{array}{cc} 0 & (\frac{t}{\sqrt k}) \rho_{1,2} \\ 0 & 0\end{array}\right].
\]
This map, $\theta'_1$, is a homomorphism with $\|[\theta'_1(\mathfrak s_1) \cdots \theta'_1(\mathfrak s_n)]\| \leq 1$ which implies by the universal property of the tensor algebra that $\theta'_1$ is completely contractive. But then $\|(1 + \frac{t}{\sqrt k}) \rho_{1,2}\| \leq 1$ which implies $\|\rho_{1,2}\| < 1$, a contradiction. 
Assume then, that 
\[
\|[ \rho_{1,1}(\mathfrak s_{i_1}) \ \rho_{1,2}(\mathfrak s_{i_1}) \cdots   \rho_{1,1}(\mathfrak s_{i_k}) \ \rho_{1,2}(\mathfrak s_{i_k}) ] \|
=  \|[\theta_1(\mathfrak s_{i_1}) \cdots \theta_1(\mathfrak s_{i_k})]\| = 1. 
\]
However, $[\rho(\mathfrak s_1) \cdots \rho(\mathfrak s_n)]$ is a row contraction as well thus $\rho_{1,3}(\mathfrak s_i) = 0$ for $1\leq i\leq n$, contradicting the fact that $\rho_{1,3}$ is a norm 1 functional. Therefore, when the ranges of the $\sigma_i$ overlap then the tensor and semicrossed product algebras cannot be completely isometrically isomorphic.

Conversely, suppose that the ranges of the $\sigma_i$ are pairwise disjoint. Because $\sigma_i$ is proper and continuous then $\sigma_i(X)$ is closed in $X$. Hence, there exists $f_1,\cdots, f_n \in C_0(X)$ such that $f_i f_j = 0$ for $i\neq j$ and $f_i(x) = 1, \forall x\in \sigma_i(X)$.
If $\mathfrak s_1,\cdots, \mathfrak s_n$ are the generators of $C_0(X) \times_\sigma \mb F_n^+$ then for $i\neq j$
\[
\mathfrak s^*_j \mathfrak s_i \ = \ (f_j\circ \sigma_j) \mathfrak s_j^* \mathfrak s_i (f_i\circ\sigma_i)
\ = \ \mathfrak s_j^* f_j f_i \mathfrak s_i \ = \ 0.
\]
Thus, $[\mathfrak s_1 \cdots \mathfrak s_n]$ is a row contraction which implies that there exists a completely contractive homomorphism from $\cA(X,\sigma)$ onto $C_0(X) \times_\sigma \mb F_n^+$ taking generators to generators. This map is the inverse of the canonical quotient from the semicrossed product onto the Tensor algebra and therefore the two algebras are completely isometrically isomorphic.
\end{proof}


The following example is due to Duncan \cite{Duncan2}, it illustrates that there are topologically isomorphic tensor and semicrossed product algebras that are not completely isometrically isomorphic.

\begin{example}
Consider the four point set $X = \{1,2,3,4\}$ and the pair of maps 
\[
\sigma_1(1) = 2, \sigma_1(2) = \sigma_1(3) = \sigma_1(4) = 3 \ \  {\rm and} \ \ \sigma_2(1) = 2, \sigma_2(2) = \sigma_2(3) = \sigma_2(4) = 4.
\]
Thus, $\cA(X,\sigma)$ and $C(X) \times_\sigma \mb F_2^+$ are isomorphic as Banach algebras simply by sending the generators to the generators. Indeed, by the previous theorem we already know that for $X' = \{2,3,4\}$, $\cA(X', \sigma) = C(X') \times_\sigma \mb F_2^+$. Now if $U_1, U_2$ are the generators of $A(\mb D) * A(\mb D)$ and $L_1,L_2$ are the generators of $\cA_2$, then 
\[
\left[\begin{array}{cc}  I & \\ &0   \end{array}\right] \mapsto \left[\begin{array}{cc}  I & \\ & 0  \end{array}\right], \ \ \left[\begin{array}{cc} 0 &  \\ & I   \end{array}\right] \mapsto \left[\begin{array}{cc}  0 & \\ & I   \end{array}\right]
\]
\[
{\rm and} \ \ \left[\begin{array}{cc}  0& U_i \\ & 0  \end{array}\right] \mapsto \left[\begin{array}{cc}  0& L_i \\ & 0   \end{array}\right], i=1,2,
\]
is an isomorphism between the algebras they generate.

It should be noted that the two maps have overlapping range so the two algebras are not (completely) isometrically isomorphic.
\end{example}




\begin{bibdiv}
\begin{biblist}


\bib{Ando}{article}{
   author={And{\^o}, T.},
   title={On a pair of commutative contractions},
   journal={Acta Sci. Math. (Szeged)},
   volume={24},
   date={1963},
   pages={88--90},
   issn={0001-6969},
}


\bib{Arv1}{article}{
   author={Arveson, W.},
   title={Operator algebras and measure preserving automorphisms},
   journal={Acta Math.},
   volume={118},
   date={1967},
   pages={95--109},
   issn={0001-5962},
}


\bib{Arv2}{article}{
   author={Arveson, W.},
   title={Subalgebras of $C\sp *$-algebras. III. Multivariable operator
   theory},
   journal={Acta Math.},
   volume={181},
   date={1998},
   number={2},
   pages={159--228},
   issn={0001-5962},
}


\bib{ArvJos}{article}{
   author={Arveson, W.},
   author={Josephson, K.},
   title={Operator algebras and measure preserving automorphisms. II},
   journal={J. Functional Analysis},
   volume={4},
   date={1969},
   pages={100--134},
}


\bib{Blech}{article}{
   author={Blecher, D.},
   title={Factorizations in universal operator spaces and algebras},
   journal={Rocky Mountain J. Math.},
   volume={27},
   date={1997},
   number={1},
   pages={151--167},
   issn={0035-7596},
}


\bib{BlechPaul}{article}{
   author={Blecher, D.},
   author={Paulsen, V.},
   title={Explicit construction of universal operator algebras and
   applications to polynomial factorization},
   journal={Proc. Amer. Math. Soc.},
   volume={112},
   date={1991},
   number={3},
   pages={839--850},
   issn={0002-9939},
}


\bib{DavFulKak}{article}{
	author={Davidson, K.},
	author={Fuller, A.},
	author={Kakariadis, E.},
	title={Semicrossed products of operator algebras: a survey},
	journal={New York J. of Math.}
	status={to appear},
}


\bib{DavFulKak2}{article}{
	author={Davidson, K.},
	author={Fuller, A.},
	author={Kakariadis, E.},
	title={Semicrossed products of operator algebras by semigroups},
	journal={Mem. Amer. Math. Soc.},
	status={to appear},
}


\bib{DavKat1}{article}{
   author={Davidson, K.},
   author={Katsoulis, E.},
   title={Isomorphisms between topological conjugacy algebras},
   journal={J. Reine Angew. Math.},
   volume={621},
   date={2008},
   pages={29--51},
   issn={0075-4102},
}


\bib{DavKat2}{article}{
   author={Davidson, K.},
   author={Katsoulis, E.},
   title={Operator algebras for multivariable dynamics},
   journal={Mem. Amer. Math. Soc.},
   volume={209},
   date={2011},
   number={982},
   pages={viii+53},
   issn={0065-9266},
   isbn={978-0-8218-5302-3},
}


\bib{DavPitts}{article}{
   author={Davidson, K.},
   author={Pitts, D.},
   title={The algebraic structure of non-commutative analytic Toeplitz
   algebras},
   journal={Math. Ann.},
   volume={311},
   date={1998},
   number={2},
   pages={275--303},
   issn={0025-5831},
}


\bib{Duncan0}{article}{
   author={Duncan, B.},
   title={$C\sp *$-envelopes of universal free products and semicrossed
   products for multivariable dynamics},
   journal={Indiana Univ. Math. J.},
   volume={57},
   date={2008},
   number={4},
   pages={1781--1788},
   issn={0022-2518},
}


\bib{Duncan1}{article}{
   author={Duncan, B.},
   title={Certain free products of graph operator algebras},
   journal={J. Math. Anal. Appl.},
   volume={364},
   date={2010},
   number={2},
   pages={534--543},
   issn={0022-247X},
}


\bib{Duncan2}{article}{
	author={Duncan, B.},
	title={Graph theoretic invariants for operator algebras associated to topological dynamics},
	journal={Rocky Mountain J. Math.},
	status={to appear},
}


\bib{DuncanPeters}{article}{
	author={Duncan, B.},
	author={Peters, J.},
	title={Operator algebras and representation from commuting semigroup actions},
	status={preprint}
	eprint={arXiv:1008.2244 [math.OA]},
}
	



\bib{FowlerRaeburn}{article}{
   author={Fowler, N.},
   author={Raeburn, I.},
   title={The Toeplitz algebra of a Hilbert bimodule},
   journal={Indiana Univ. Math. J.},
   volume={48},
   date={1999},
   number={1},
   pages={155--181},
   issn={0022-2518},
}


\bib{HadHoo}{article}{
   author={Hadwin, D.},
   author={Hoover, T.},
   title={Operator algebras and the conjugacy of transformations},
   journal={J. Funct. Anal.},
   volume={77},
   date={1988},
   number={1},
   pages={112--122},
   issn={0022-1236},
}


\bib{KakKats}{article}{
   author={Kakariadis, E.},
   author={Katsoulis, E.},
   title={Isomorphism invariants for multivariable ${\rm C}\sp
   \ast$-dynamics},
   journal={J. Noncommut. Geom.},
   volume={8},
   date={2014},
   number={3},
   pages={771--787},
   issn={1661-6952},
}

\bib{KatsKribs}{article}{
   author={Katsoulis, E.},
   author={Kribs, D.},
   title={Isomorphisms of algebras associated with directed graphs},
   journal={Math. Ann.},
   volume={330},
   date={2004},
   number={4},
   pages={709--728},
   issn={0025-5831},
}


\bib{Lig}{article}{
   author={Ligocka, E.},
   title={On proper holomorphic and biholomorphic mappings between product
   domains},
   language={English, with Russian summary},
   journal={Bull. Acad. Polon. Sci. S\'er. Sci. Math.},
   volume={28},
   date={1980},
   number={7-8},
   pages={319--323 (1981)},
   issn={0137-639X},
}


\bib{MuhlySolel}{article}{
   author={Muhly, P.},
   author={Solel, B.},
   title={Tensor algebras over $C\sp *$-correspondences: representations,
   dilations, and $C\sp *$-envelopes},
   journal={J. Funct. Anal.},
   volume={158},
   date={1998},
   number={2},
   pages={389--457},
   issn={0022-1236},
}


\bib{Parrott}{article}{
   author={Parrott, S.},
   title={Unitary dilations for commuting contractions},
   journal={Pacific J. Math.},
   volume={34},
   date={1970},
   pages={481--490},
   issn={0030-8730},
}


\bib{Peters}{article}{
   author={Peters, J.},
   title={Semicrossed products of $C\sp \ast$-algebras},
   journal={J. Funct. Anal.},
   volume={59},
   date={1984},
   number={3},
   pages={498--534},
   issn={0022-1236},
}


\bib{Plastiras}{article}{
   author={Plastiras, J.},
   title={$C\sp*$-algebras isomorphic after tensoring},
   journal={Proc. Amer. Math. Soc.},
   volume={66},
   date={1977},
   number={2},
   pages={276--278},
   issn={0002-9939},
}


\bib{Pop0}{article}{
   author={Popescu, G.},
   title={von Neumann inequality for $(B({\scr H})\sp n)\sb 1$},
   journal={Math. Scand.},
   volume={68},
   date={1991},
   number={2},
   pages={292--304},
   issn={0025-5521},
}


\bib{Pop1}{article}{
   author={Popescu, G.},
   title={Noncommutative joint dilations and free product operator algebras},
   journal={Pacific J. Math.},
   volume={186},
   date={1998},
   number={1},
   pages={111--140},
   issn={0030-8730},
}


\bib{Pop2}{article}{
   author={Popescu, G.},
   title={Universal operator algebras associated to contractive sequences of
   non-commuting operators},
   journal={J. London Math. Soc. (2)},
   volume={58},
   date={1998},
   number={2},
   pages={469--479},
   issn={0024-6107},
}


\bib{Pop3}{article}{
   author={Popescu, G.},
   title={Free holomorphic automorphisms of the unit ball of $B(\scr H)\sp
   n$},
   journal={J. Reine Angew. Math.},
   volume={638},
   date={2010},
   pages={119--168},
   issn={0075-4102},
}


\bib{Raeburn}{book}{
   author={Raeburn, I.},
   title={Graph algebras},
   series={CBMS Regional Conference Series in Mathematics},
   volume={103},
   publisher={Published for the Conference Board of the Mathematical
   Sciences, Washington, DC; by the American Mathematical Society,
   Providence, RI},
   date={2005},
   pages={vi+113},
   isbn={0-8218-3660-9},
}

\bib{Ramsey}{thesis}{
	author={Ramsey, C.},
	title={Algebraic characterization of multivariable dynamics},
	organization={University of Waterloo},
	type={Master's Thesis},
	date={2009},
}



\bib{Rudin}{book}{
   author={Rudin, W.},
   title={Function theory in polydiscs},
   publisher={W. A. Benjamin, Inc., New York-Amsterdam},
   date={1969},
   pages={vii+188},
}


\bib{Rudin2}{book}{
   author={Rudin, W.},
   title={Function theory in the unit ball of ${\bf C}\sp{n}$},
   series={Grundlehren der Mathematischen Wissenschaften [Fundamental
   Principles of Mathematical Science]},
   volume={241},
   publisher={Springer-Verlag, New York-Berlin},
   date={1980},
   pages={xiii+436},
   isbn={0-387-90514-6},
}

\bib{Solel}{article}{
   author={Solel, B.},
   title={You can see the arrows in a quiver operator algebra},
   journal={J. Aust. Math. Soc.},
   volume={77},
   date={2004},
   number={1},
   pages={111--122},
   issn={1446-7887},
}


\bib{NagyFoias}{book}{
   author={Sz.-Nagy, B.},
   author={Foias, C.},
   author={Bercovici, H.},
   author={K{\'e}rchy, L.},
   title={Harmonic analysis of operators on Hilbert space},
   series={Universitext},
   edition={2},
   edition={Revised and enlarged edition},
   publisher={Springer, New York},
   date={2010},
   pages={xiv+474},
   isbn={978-1-4419-6093-1},
}



\bib{Tsyganov}{article}{
   author={Tsyganov, Sh.},
   title={Biholomorphic mappings of the direct product of domains},
   language={Russian},
   journal={Mat. Zametki},
   volume={41},
   date={1987},
   number={6},
   pages={824--828, 890},
   issn={0025-567X},
}


\bib{Varopoulos}{article}{
   author={Varopoulos, N. Th.},
   title={On an inequality of von Neumann and an application of the metric
   theory of tensor products to operators theory},
   journal={J. Functional Analysis},
   volume={16},
   date={1974},
   pages={83--100},
}


\bib{Voic}{article}{
   author={Voiculescu, D.},
   title={Symmetries of some reduced free product $C\sp \ast$-algebras},
   conference={
      title={Operator algebras and their connections with topology and
      ergodic theory},
      address={Bu\c steni},
      date={1983},
   },
   book={
      series={Lecture Notes in Math.},
      volume={1132},
      publisher={Springer, Berlin},
   },
   date={1985},
   pages={556--588},
}


\end{biblist}
\end{bibdiv}
\end{document}